\documentclass[12pt, reqno]{amsart}

\usepackage[all, arc]{xy}
\usepackage{enumerate}
\usepackage{mathrsfs}
\usepackage{hyperref}
\usepackage{stmaryrd}
\usepackage{scalerel}
\usepackage{comment}
\usepackage[all]{xy}
\usepackage{amsthm}
\usepackage{tipa}
\usepackage{verbatim}
\usepackage{tikz-cd}
\usepackage{tree}
\usepackage{tikz}
\usepackage{multirow}

\usepackage[latin1]{inputenc}
\usepackage{babel}
\usepackage{amsmath}
\usepackage[T1]{fontenc}
\usepackage[T3,OT2,T1]{fontenc}
\usepackage{textcomp}
\usepackage{bm}
\addto{\captionsafrikaans}{}
\addto{\captionsenglish}{}
\usepackage{graphicx}
\usepackage{color}
\usepackage{eso-pic}
\usepackage{amssymb}
\usepackage{mathtools}
\usepackage{multicol}
\usepackage{paralist}
\usepackage{geometry}
\usepackage{algorithmic}
\usepackage[ruled,vlined]{algorithm2e}
\usepackage{amsfonts}
\usepackage{filecontentsdef}
\usepackage{xcolor}
\usepackage{tikz-3dplot}
\usepackage{url}
\usepackage{graphicx}
\usetikzlibrary{decorations.pathmorphing}
\usepackage{cooltooltips}
\usepackage{setspace}

\DeclareOldFontCommand{\bf}{\normalfont\bfseries}{\mathbf}

\newtheorem{theorem}{Theorem}[section]
\newtheorem{corollary}[theorem]{Corollary}
\newtheorem{lemma}[theorem]{Lemma}
\newtheorem{proposition}[theorem]{Proposition}
\theoremstyle{definition}
\newtheorem{definition}[theorem]{Definition}
\newtheorem{remark}[theorem]{Remark}
\newtheorem{example}[theorem]{Example}

\newcommand{\oid}[1]{\operatorname{id}}

\newcommand{\tr}[1]{\mathfrak{t}^{\pm}_{\zeta_{#1}}}
\newcommand{\ts}[1]{\mathfrak{t}^{+}_{\zeta_{#1}}}
\newcommand{\tq}[1]{\mathfrak{t}^{-}_{\zeta_{#1}}}

\newcommand\restr[2]{{
\left.\kern-\nulldelimiterspace 
#1 
\vphantom{\big|} 
\right|_{#2} 
}}



\title[A study of a recursive sequence of polynomials]{\bfseries A study of a recursive sequence of polynomials revealing weighted Catalan Numbers}
\author{Sophie Marques and Elizabeth Mrema}

\begin{document}

  \setcounter{tocdepth}{3}
\maketitle
\begin{center}
\rm e-mail: smarques@sun.ac.za

\it
Department of Mathematical Sciences, 
Stellenbosch University, \\
Stellenbosch, 7600, 
South Africa\\
\&
NITheCS (National Institute for Theoretical and Computational Sciences), \\
South Africa \\ \bigskip

\rm e-mail: mrema.elizabeth@udsm.ac.tz

\it
Department of Mathematics, University of Dar es Salaam\\
Tanzania. 
\end{center} 
    \tableofcontents

\begin{abstract}

This paper examines the recursive sequence of polynomials \(p_n(x)\), defined by \(p_0(x) = x^2 - 2\) and \(p_n(x) = p_{n-1}(x)^2 - 2\) for \(n \geq 1\). It describes the field-theoretic motivations behind this sequence, derives a recursive formula for its coefficients, and identifies invariants that uncover combinatorial connections, including links to weighted Catalan numbers.  \\ \\

\vspace*{\fill}


\end{abstract}

\section*{Introduction}

Recursive polynomial sequences such as Laguerre, Chebyshev, Legendre, Jacobi, and Hermite polynomials have been extensively studied in the literature (see, for instance, \cite{aziz2020orthogonal, benjamin2009counting, doha2004construction, doha2004connection, hone2018family}). These sequences have applications in physics, statistics, and approximation theory. Although this paper does not focus on these applications, we dedicate our study to a specific recursive sequence of polynomials defined by setting:

$$p_0(x) = x^2 - 2, \quad \text{and for all} \ n \geq 1, \ p_n(x) = p_{n-1}(x)^2 - 2.$$

The polynomials \(p_n(x)\) can express the minimal polynomials of the generators of the forms \(\zeta_{2^e} + \zeta_{2^e}^{-1}\) of a codegree 2 subextensions of the cyclotomic extension \(F(\zeta_{2^e})/F\) for some \(e \in \mathbb{N}\) (see \cite[Theorem 4.1]{marques-cyclic}). This provides a field-theoretic motivation for the focus of this paper.

On one hand, we derive a recursive formula for the coefficients of each term in this sequence by leveraging the field-theoretic interpretation mentioned above (see Theorem \ref{coefficients-field}). This approach opens the door to potentially applying field-theoretic methods to compute coefficients of other polynomial recursive sequences. On the other hand, a detailed examination of these coefficients leads us to our second main result (see Theorem \ref{coefficients-ai}), where we express the coefficients \(c_{n, 2k}\) as follows:

$$c_{n,2k} = \begin{cases}
2 & \text{when } k = 0, \\
\sum\limits_{j=1}^{k} a_{j,k} 2^{2jn} & \text{otherwise}
\end{cases}$$

Here, \(a_{j,k}\) is a rational number independent of \(n\), providing us with invariants for the sequence $(p_n(x))_{n \in \mathbb{N}}$.

Given these invariants, we sought to understand their significance and explore whether they possessed interesting combinatorial properties. Leveraging the earlier computation of the coefficients \(c_{n,2k}\), we devised an algorithmic method (see Proposition \ref{vandermond}) to compute these invariants \(a_{j,k}\) using a Vandermonde approach, similar to the one employed in \cite{pereyra-vandermonde}.

Through this process, we isolated the invariant \(a_{k,k}\). Further analysis and observations led us to identify these \(a_{k,k}\) as weighted Catalan numbers, providing a novel method for their computation. However, the combinatorial significance of the remaining invariants \(a_{j,k}\) remains unclear, presenting an open question that emerges from this study.

\section*{Notation and symbols}
In this paper, $F$ denotes a field, $\overline{F}$ is one fixed algebraic closure of the field $F$,  and $e$ is a positive integer. We denote by $\zeta_{2^e}$ one primitive $2^e$ root of unity in $\overline{F}$.

In the following table, we present notation and symbols that will be utilized throughout this paper.
\begin{center} 
  \begin{tabular}{l p{11cm} }
$\mathbb{N}$ & The set of natural number including zero.\\
$\llbracket  m, n  \rrbracket$ & $\{ m, m+1,\cdots , n \}$ where $m, n \in \mathbb{N}$ such that $m \leq n$.\\

$\tr{2^e}$ & $\zeta_{2^e} \pm \zeta_{2^e}^{-1}$.\\
$o_F(\alpha)$ & The order of an element $\alpha\in \overline{F}$ over a field $F$.\\

 $\operatorname{t}_F(2^e)$ & $\left\{ \begin{array}{clll} 
 	2^e & \text{when } \operatorname{o}_F(\zeta_{2^e})>2;\\
 	2 & \text{when }  \operatorname{o}_F(\zeta_{2^e})=2;\\
 	1 & \text{when } \operatorname{o}_F(\zeta_{2^e})=1.
 \end{array} \right.$ \\
 $\nu_{2^\infty_F}^{+}$ & $\left\{ \begin{array}{clll}\operatorname{max}\{k \in \mathbb{N} | \zeta_{t_F(2^k)} +
 	\zeta_{t_F(2^k)}^{-1} \in F\} & \text{ when it exists }; \\
 	\infty & \text{ otherwise.} \end{array} \right.$\\
 $\nu_{2^\infty_F}$ & $\left\{ \begin{array}{clll} \nu_{2^\infty_F}^++1 & \text{when}\ F \ \text{has property } \ \mathcal{C}_2 \ (\text{see \cite[Definition 3.13]{marques-quadratic}}) ; \\ 
 	\nu_{2^\infty_F}^+ & \text{ otherwise. }  \end{array} \right.$ 
  \end{tabular}
\end{center}

\section{The recursive sequence of polynomials as minimal polynomials}
\subsection{Definition}
We begin this section by introducing the recursively defined sequence of polynomials that will be the focus of this paper.

\begin{definition}[Lemma]\label{recursive-polynomial}
    We define the sequence of polynomials \((p_n(x))_{n \in \mathbb{N}}\) as follows:
    \[
    p_0(x) = x^2 - 2 \quad \text{and for all} \ n \geq 1, \quad p_n(x) = p_{n-1}(x)^2 - 2.
    \]
    For any \(n, k \in \mathbb{N}\), let \(c_{n,k}\) denote the coefficient of \(x^k\) in the polynomial \(p_n(x)\).
\end{definition}

\begin{remark}\label{recursive-remark}
    Let \(n \in \mathbb{N}\). 
    \begin{enumerate}
        \item All odd coefficients of \(p_n(x)\) are zero, so that
        \[
        p_n(x) = \sum_{k=0}^{2^n} c_{n,2k} x^{2k}.
        \]
      
        \item We observe that \(p_0(0) = -2\), and for all \(n \geq 1\), \(p_n(0) = 2\). This statement can be established through a straightforward induction.  
        \item The polynomial \(p_n(x)\) has coefficients in \(\mathbb{Z}\). In this paper, we will also consider these polynomials with coefficients in an arbitrary field \(F\). More precisely, given the characteristic map \(\chi_F: \mathbb{Z} \to F\), which sends \(a\) to \(a \cdot 1_F\), and a polynomial \(p(x) = \sum_{i=0}^{n} a_i x^i\) in \(\mathbb{Z}[x]\), we consider \(p(x)\) as a polynomial in \(F[x]\) by identifying it with \(\sum_{i=0}^{n} \chi_F(a_i) x^i\).
    \end{enumerate}
\end{remark}

\begin{example}
The polynomial \( p_n(x) \) for \( n \in \{1, 2, 3, 4\} \) are as follows:
    \begin{align*}
        p_1(x) &= x^4 - 4x^2 + 2, \\
        p_2(x) &= x^8 - 8x^6 + 20x^4 - 16x^2 + 2, \\
        p_3(x) &= x^{16} - 16x^{14} + 104x^{12} - 352x^{10} + 660x^8 - 672x^6 + 336x^4 - 64x^2 + 2, \\
        p_4(x) &= x^{32} - 32x^{30} + 464x^{28} - 4032x^{26} + 23400x^{24} - 95680x^{22} + 283360x^{20} \\
        & \quad - 615296x^{18} + 980628x^{16} - 1136960x^{14} + 940576x^{12} - 537472x^{10} \\
        & \quad + 201552x^8 - 45696x^6 + 5440x^4 - 256x^2 + 2.
    \end{align*}
\end{example}

\noindent We now define a new sequence of polynomials that will ultimately coincide with the previously defined sequence.
\begin{definition}\label{recursive-q}
    We introduce a new sequence of polynomials \((q_n(x))_{n \in \mathbb{N}}\) defined recursively as follows:
    \begin{equation*}
        q_0(x) = x^2 - 2, \text{ and for all  \(n \geq 1\), }
        q_n(x) = q_{n-1}(x^2 - 2).
    \end{equation*}
\end{definition}

\subsection{Characterization of $p_n(x)$ as minimal polynomial}
The following lemma establishes the minimal polynomial of ${\ts{2^e}}$ using the previously constructed sequences of polynomials.
\begin{proposition}\label{minpn}
	Let  $n\in \mathbb{N}\cup \{0\}$. 
	The following assertions are equivalent:
	\begin{enumerate}
		\item $[F({\ts{2^e}}):F]=2^{n+1}$
		\item $\operatorname{min}({\ts{2^e}}, F)=p_n(x)-\ts{2^{\scalebox{0.5}{$\nu_{2^\infty_F}^+$}}}$
		\item $\operatorname{min}({\ts{2^e}}, F)=q_n(x)-\ts{2^{\scalebox{0.5}{$\nu_{2^\infty_F}^+$}}}$
	\end{enumerate}
When one of these assumption is satisfied, we have $e=n+(\nu_{2^\infty_F}^++1)$.
\end{proposition}
\begin{proof}
We start our proof by showing that \((1)\) implies \((2)\) and that \(e = n + (\nu_{2^\infty_F}^+ + 1)\).  
Assume that \([F(\ts{2^e}): F] = 2^{n+1}\). Our goal is to demonstrate that \(\operatorname{min}(\ts{2^e}, F) = p_n(x) - \ts{2^{\scalebox{0.5}{$\nu_{2^\infty_F}^+$}}}\) and \(e = n + (\nu_{2^\infty_F}^+ + 1)\). We will use an induction argument on \(n\) to establish this result.
 
For \(n = 0\), we have \([F(\ts{2^e}): F] = 2\). According to \cite[Theorem 3.2]{marques-cyclic}, we know that \(\operatorname{min}(\ts{2^e}, F) = x^2 - (\ts{2^{e-1}} + 2) = p_0(x) - \ts{2^{e-1}}\). Additionally, since \(\ts{2^{e-1}} \in F\) and \(\ts{2^e} \notin F\), we conclude that \(e = \nu_{2^\infty_F}^+ + 1\).

Assume that the statement holds for some \(n \geq 1\) over any base field. We need to prove that it also holds for \(n + 1\). We suppose that \([F(\ts{2^e}): F] = 2^{n+2}\) and we set \(M := F(\ts{2^{\scalebox{0.5}{$\nu_{2^\infty_F}^+ + 1$}}})\). By \cite[Theorem 3.2]{marques-cyclic}, the field extension \(F(\ts{2^e})/F\) can be expressed as a tower \(F(\ts{2^e}) / M / F\) and $[M: F]=2$. Thus, by the multiplicativity of the degree of a tower of field extensions, we have \([F(\ts{2^e}): M] = 2^{n+1}\).  Applying the induction hypothesis, we conclude that \(\operatorname{min}(\ts{2^e}, M) = p_n(x) - \ts{2^{\scalebox{0.5}{$\nu_{2^\infty_M}^+$}}}\) and \( e =  n+ (\nu_{2^\infty_M}^+ + 1)\). Moreover, \(\nu_{2^\infty_M}^+ = \nu_{2^\infty_F}^+ + 1\). Therefore, \( e =  n+ (\nu_{2^\infty_F}^+ + 2)\) and \(p_n(\ts{2^e}) = \ts{2^{\scalebox{0.5}{$\nu_{2^\infty_F}^+ + 1$}}}\). Squaring this equation, we obtain \(p_n(\ts{2^e})^2 = \ts{2^{\scalebox{0.5}{$\nu_{2^\infty_F}^+$}}} + 2\), or equivalently, \(p_{n+1}(\ts{2^e}) -\ts{2^{\scalebox{0.5}{$\nu_{2^\infty_F}^+$}}}= 0\). Given that \(\ts{2^{\scalebox{0.5}{$\nu_{2^\infty_F}^+$}}} \in F\) by the definition of \(\nu_{2^\infty_F}^+\), and that \(p_{n+1}(x) - \ts{2^{\scalebox{0.5}{$\nu_{2^\infty_F}^+$}}}\) has degree \(2^{n+2}\), while \([F(\ts{2^e}): F] = 2^{n+2}\) by assumption, we deduce that \(p_{n+1}(x) - \ts{2^{\scalebox{0.5}{$\nu_{2^\infty_F}^+$}}} = \operatorname{min}(\ts{2^e}, F)\). This completes the induction.

$(1) \Longrightarrow (3)$
The proof follows similarly by induction on \(n\), using the same approach as in the implication \((1) \Longrightarrow (2)\). The only distinction lies in the transmission step of the induction, where we utilize the tower \(F(\ts{2^e}) / F(\ts{2^{e-1}}) / F\), apply the induction hypothesis to \(F(\ts{2^{e-1}}) / F\)  and use the fact that $\ts{2^{e-1}}= {\ts{2^{e}}}^2-2$ instead.  

$(2) \Longrightarrow (1)$ and $(3) \Longrightarrow (1)$
These implications follow directly since both \(p_n(x) - \ts{2^{\scalebox{0.5}{$\nu_{2^\infty_F}^+$}}}\) and \(q_n(x) - \ts{2^{\scalebox{0.5}{$\nu_{2^\infty_F}^+$}}}\) have degree \(2^{n+1}\). Thus, the proof is complete.
\end{proof}

\begin{remark}\label{rktp}
From Proposition \ref{minpn}, we deduce the following statements:
	\begin{enumerate}
		\item We have established that \(\operatorname{min}(\ts{2^e}, F) = p_{e - (\nu_{2^\infty_F}^+ + 1)}(x) - \ts{2^{\scalebox{0.5}{$\nu_{2^\infty_F}^+$}}}\).
		\item In the special case where \(F = \mathbb{Q}\), we have \(\nu_{2^\infty_F}^+ = 2\) and \(\ts{2^{\scalebox{0.5}{$\nu_{2^\infty_F}^+$}}} = 0\). Consequently, \(\operatorname{min}(\ts{2^{n+3}}, \mathbb{Q}) = p_n(x)\), and \(p_n(x)\) is an irreducible polynomial over \(\mathbb{Q}\), for all \(n \in \mathbb{N}\).
		\item We have shown that the polynomials \(p_n(x)\) and \(q_n(x)\) are equal for all \(n \in \mathbb{N}\).
	\end{enumerate}
\end{remark}
\begin{corollary}\label{minimal-negative}
 We have $$ \operatorname{min}(\tq{2^e}, F)=\begin{cases}
		p_{e-(\nu_{2^\infty_F}^++1)}(x^2+2)-\ts{2^{\scalebox{0.5}{$\nu_{2^\infty_F}^+$}}}   &  \text{if} \ e>\nu_{2^\infty_F}^++1 \\
		x^2-(\ts{2^{\scalebox{0.5}{$\nu_{2^\infty_F}^+$}}}-2) & \text{if} \ e=\nu_{2^\infty_F}^++1.
	\end{cases} $$ 

\end{corollary}
\begin{proof}
We have that $\tq{2^e}$ is a root of the polynomial $p(x^2+2)$ where $p(x)=\operatorname{min}(\ts{2^{e-1}}, F)$. Indeed, ${\tq{2^e}}^2+2=(\zeta_{2^e}-\zeta_{2^e}^{-1})^2+2=\ts{2^{e-1}}$.  That implies that $p({\tq{2^e}}^2+2)=p(\ts{2^{e-1}})=0$. Thus the result follows trivially when $e=\nu_{2^\infty_F}^++1$ and follows from Proposition \ref{minpn} otherwise.
\end{proof}

\section{Computing the coefficient of $p_n(x)$ using field theory}
In the following result, we compute the coefficients of the recursive polynomial $p_n(x)$ using field theory.
\begin{theorem}\label{coefficients-field}
 
The coefficients $c_{n, 2(2^{n}-j)}$ of $p_n(x)$ are defined recursively on $j$ from $0$ to $2^{n}$ as follows: $c_{n, 2^{n+1}}=1, \ 
	\text{and for all} \ j \in \llbracket  1, 2^{n} \rrbracket,$
	$$c_{n, 2(2^{n}-j)} = -\sum \limits^{2^{n}}_{i=2^{n}-j+1}\binom{2i}{j+i-2^{n}}c_{n,2i}.
	$$
\end{theorem}

\begin{proof}
	
	We recall that $[\mathbb{Q}(\zeta_{2^{n+3}}):\mathbb{Q}]=2^{n+2}$. 
	Moreover, by Remark \ref{rktp}, $\operatorname{min}(\ts{2^{n+3}}, \mathbb{Q})=p_n(x)$ and $[\mathbb{Q}(\ts{2^{n+3}}):\mathbb{Q}]=\operatorname{deg}(p_n(x))=2^{n+1}$.
	
Now, let us define $r(x):=x^{2^{n+1}}p_n\big(x+ \frac{1}{x}\big)$. It is clear that $r(x)$ is a polynomial over $\mathbb{Q}$, and $r(\zeta_{2^{n+2}})=\zeta_4p_n(\ts{2^{n+2}})=0$.
	Since $[\mathbb{Q}(\zeta_{2^{n+3}}):\mathbb{Q}]=2^{n+2}$ 
	
	and $\operatorname{deg}(\operatorname{min}(\zeta_{2^{n+3}}, \mathbb{Q}))= \operatorname{deg} (r(x))= 2^{n+2}$,
	we can deduce that $r(x)=\operatorname{min}(\zeta_{2^{n+3}}, \mathbb{Q})$.

	By \cite[Corollary 4.2 ]{marques-cyclic}, we know that  $\operatorname{min}(\zeta_{2^{n+3}}, \mathbb{Q})=x^{2^{n+2}}+1$. 
	
	As a result, we get that
	\begin{align}\label{eq2}
		r(x)= x^{2^{n+1}}\sum \limits^{2^{n}}_{i=0}c_{n,2i}\left(x+\frac{1}{x}\right)^{2i}=x^{2^{n+2}}+1  
	\end{align} 
	We have

	\begin{alignat*}{2}
	 x^{2^{n+1}}\sum \limits^{2^{n}}_{i=0}c_{n,2i}\left(x+\frac{1}{x}\right)^{2i}
	&
		= \sum \limits^{m}_{i=0}\sum \limits^{2i}_{k=0}\binom{2i}{k}c_{n,2i}x^{2(k-i+m)}\\
		&= \sum \limits^{m}_{i=0}\sum \limits^{i+m}_{d=m-i}\binom{2i}{d+i-m}c_{n,2i}x^{2d}\\
	&=\sum \limits^{2m}_{d=0}\sum \limits^{m}_{i=|m-d|}\binom{2i}{d+i-m}c_{n,2i}x^{2d}
	\end{alignat*}

	$	\text{where} \ m=2^{n} , \ d = k-i+m$ and the last equality is obtained by exchanging the sums in $d$ and $i$.
	
	Thus, Equation (\ref{eq2}) becomes 
	\begin{align}\label{eq3}
		\sum \limits^{2m}_{d=0}\sum \limits^{m}_{i=|m-d|}\binom{2i}{d+i-m}c_{n,2i}x^{2d} =x^{4m}+1 
	\end{align}
	For any $i\in \llbracket  0 , 2^{n} \rrbracket,$ and $d\in \llbracket  0 , 2^{n+1} \rrbracket,$ we set $b(i,d) := \binom{2i}{d+i-m}$. Comparing the coefficients of the two polynomials and using the symmetry of the binomial coefficients, we find that solving Equation (\ref{eq3}) for the $c_{n, 2j}$'s is equivalent to solving the matrix equation:

	\[
\small
	\begin{bmatrix}
		1 & b(1,m) & b(2,m) & \cdots  & b(m-1, m) & b(m,m)\\
		0 & 1 & b(2, m-1) & \cdots & b(m-1, m-1) & b(m,m-1)\\
		0 &0 & 1 & \cdots  & b(m-1, m-2) & b(m,m-2)\\
		\vdots & \vdots & \vdots & \vdots  & \vdots & \vdots  \\
		0 &  0 & 0 &\cdots   & b(m-1, 2) & b(m, 2)\\
		0 & 0 & 0&\cdots & 1 & b(m, 1)\\
		0 & 0 & 0& \cdots  &0 & 1\\
	\end{bmatrix}
	\begin{bmatrix}
		c_{n,0} \\
		c_{n, 2}\\
		c_{n, 6}\\
		\vdots\\
		c_{n, 2(m-2)}\\
		c_{n, 2(m-1)}\\
		c_{n, 2m}\\
	\end{bmatrix}
	=
	\begin{bmatrix}
		0 \\
		0 \\
		0\\
		\vdots \\
		0 \\
		0\\
		1\\
	\end{bmatrix}
	\]
	Since the system above is in echelon form, we can use back substitution to solve for the system and conclude the proof of the theorem. 
\end{proof}
\begin{remark}
From Theorem \ref{coefficients-field}, for any $n \in \mathbb{N}$, we obtain
\begin{enumerate}
\item $c_{n, 2(2^n-1)} = -2^{n+1}$,  
\item $c_{n, 2 (2^n-2)}= 2^{n} ( 2^{n+1}-3)$, for any $n \geq 1$, 
\item $c_{n, 2 (2^n-3)}=2^{n+1}(2^{n+1}-3)-\frac{2^n(2^{n+1}-1)(2^{n+1}-2)}{3}$, for any $n\geq 2$,
\item $c_{n, 2 (2^n-4)} =2^{n-1} (2^n-3) \left(\frac{4(2^{n+1}-1)(2^{n+1}-2)}{3}-  (2^{2(n+1)}-9)\right),$ for any $n \geq 3$.







\item $2 = -\sum \limits^{2^{n}}_{i=1}b(i,2^n)c_{n,2i}$, for any $n \geq 1$, since $c_{n, 0} =2$, by Remark \ref{recursive-remark}, (2).
\end{enumerate}
\end{remark}

We now proceed by characterizing the coefficients of the polynomial sequence \( p_n(x) \), this time relying on its recursive definition rather than field-theoretic techniques. To achieve this, we express the coefficients of \( p_n(x) \) in terms of those of \( p_{n-1}(x) \).

We begin by introducing a key function that will facilitate the description of these coefficients.

\begin{definition}

We define the map \( \epsilon: \mathbb{N} \to \{0, 1\} \) as follows:
\[
    \epsilon(k) =
    \begin{cases}
        0, & \text{if } k \text{ is odd} \\
        1, & \text{if } k \text{ is even}.
    \end{cases}
\]
\end{definition} 
\begin{lemma}\label{coefficients}
	Let $n\geq 1$ and $k\in \llbracket 0, 2^{n} \rrbracket$. The coefficients $c_{n, 2k}$ of $p_n(x)$ are defined recursively as follows: $$c_{n,0}=c_{n-1, 0}^2-2 $$
	and
	$$c_{n,2k}=\epsilon(k)(c_{n-1,k})^2+2\sum\limits^{\tiny \text{$\operatorname{min}(\lfloor\frac{k-1}{2}\rfloor,2^{n-1}-1)$}}_{\tiny \text{$s=\operatorname{max}(0, k-2^{n-1})$}}c_{n-1,2s}c_{n-1, 2(k-s)}, \text{ for all } \  k\in \llbracket 1,  2^{n} \rrbracket $$ 
\end{lemma} 
\begin{proof}
	We set $M:=\operatorname{min}(\lfloor\frac{k-1}{2}\rfloor,2^{n-1}-1)$ and $m:= \operatorname{max}(0, k-2^{n-1})$. By Definition \ref{recursive-polynomial}, we know that $p_{n}(x)=p_{n-1}(x)^2-2$. From this equation, we can deduce
	\begin{alignat*}{2}
		p_n(x) &=\left(\sum\limits^{2^{n-1}}_{k=0}c_{n-1,2k}x^{2k}\right)^2-2\\
		&=\sum\limits^{2^{n-1}}_{k=0}(c_{n-1,2k})^2x^{2(2k)}+2\sum\limits^{2^{n-1}-1}_{s=0}\sum\limits^{2^{n-1}}_{t=s+1}c_{n-1,2s}c_{n-1,2t}x^{2(s+t)}-2  \\
		& =\sum\limits^{2^{n-1}}_{k=0}(c_{n-1,2k})^2x^{2(2k)}+2\sum\limits^{2^{n}-1}_{k=1}\sum\limits^{M}_{\tiny \text{$s=m$}}c_{n-1,2s}c_{n, 2(k-s)}x^{2k}-2.
	\end{alignat*}
	
Recalling that $$\epsilon(k)=
	\begin{cases}
	0 & \text{if } 2\nmid k \\
	1 & \text{if }  2\mid k \\
	\end{cases},$$
we obtain that
	\begin{alignat*}{2}
		\sum\limits^{2^n}_{k=0}c_{n,2k}x^{2k}&=\sum\limits^{2^{n}}_{k=0}\epsilon(k)(c_{n-1,k})^2x^{2k}+2\sum\limits^{2^{n}-1}_{k=1}\sum\limits^{M}_{s=m}c_{n-1,2s}c_{n-1, 2(k-s)}x^{2k}-2\\
		&=(c_{n-1,0})^2-2+\sum\limits^{2^n}_{k=1}\left(\epsilon(k)(c_{n-1,k})^2+2\sum\limits^{M}_{s=m}c_{n-1,2s}c_{n-1, 2(k-s)}\right)x^{2k}. 
	\end{alignat*}
	By comparing the coefficients of the two polynomials, we can conclude that the desired result holds.  
\end{proof}

We also observe that by applying the Taylor expansion to \( p_n(x) \) around \( x=0 \), the coefficients can be expressed in terms of the derivatives of \( p_n(x) \) evaluated at \( x=0 \).

Specifically, for \( k \in \llbracket 0, 3 \rrbracket \), the coefficients can be written as:
\[
    c_{n,0} = p_n(0), \quad c_{n,2} = \frac{p_n^{(2)}(0)}{2!}, \quad c_{n,4} = \frac{p_n^{(4)}(0)}{4!}, \quad c_{n,6} = \frac{p_n^{(6)}(0)}{6!}.
\]

Here, \( p_n^{(k)}(x) \) denotes the \( k \)-th derivative of \( p_n(x) \).

To compute these coefficients, we begin with the recursive definition $$p_n(x) = (p_{n-1}(x))^2 - 2 .$$ By applying the chain rule, we can derive expressions for the derivatives of \( p_n(x) \). 

For example, the first and second derivatives are:
\[
    p_n'(x) = 2 p_{n-1}(x) p_{n-1}'(x), \quad p_n''(x) = 2(p_{n-1}'(x))^2 + 2p_{n-1}(x)p_{n-1}''(x).
\]
Evaluating the second derivative at \( x = 0 \), we find that \( c_{1,2} = -4 \). 

Moreover, we observe that \( p_n''(0) = 2^2 p_{n-1}''(0) \), since all odd coefficients vanish and the constant term remains 2 for any \( n \geq 1 \). Therefore, $c_{n,2} = 2^2 c_{n-1,2}$ for $n\geq 1$, and we deduce that
	$$c_{n,2}=-2^{2n}.$$

Similarly, by considering the fourth derivative of \( p_n(x) \) and evaluating it at \( x = 0 \), we obtain \( c_{n,4} = 2^{4n-4} + 4c_{n-1,4} \). Introducing the sequence \( a_n = c_{n,4} - 4c_{n-1,4} \), we can express this as a recursive formula: \( a_n = 2^4 a_{n-1} \), with \( a_2 = 2^4 \). This recurrence relation is a geometric sequence, which allows us to derive the following recursive relation for \( c_{n,4} \): 
\[
c_{n,4} - 20 c_{n-1,4} + 6 c_{n-2,4} = 0.
\]
The associated characteristic polynomial is \( x^2 - 20x + 64 = (x - 16)(x - 4) = 0 \). Thus, the solution for \( c_{n,4} \) is of the form:
\[
c_{n,4} = a 4^{n-1} + b 16^{n-1}
\]
for some constants \( a \) and \( b \). Using the initial conditions \( c_{1,4} = a + b = 1 \) and \( c_{2,4} = 4a + 16b = 20 \), we solve these equations to find:
\[
c_{n,4} = -\frac{1}{12} 2^{2n} + \frac{1}{12} 2^{4n}.
\]

Generalizing this pattern, we derive a similar formula for \( c_{n,6} \):
\[
c_{n,6} = -\frac{1}{90} 2^{2n} + \frac{1}{72} 2^{4n} - \frac{1}{360} 2^{6n}.
\]

Having established this pattern, we formulated a general conjecture. The following theorem captures this result, and we proved our conjecture using Lemma \ref{coefficients}.

\begin{theorem}\label{coefficients-ai} 
	Let $n\geq 1$ and $k \in\llbracket 0 , 2^n\rrbracket$. The coefficients $c_{n, 2k}$ of $p_n(x)$ are defined as follows: 
	
	$$c_{n,2k}=\left\{
	\begin{array}{cl}
		2 & \text{when } k=0, \\
		\sum\limits^{k}_{j=1}a_{j,k}2^{2jn} & \text{otherwise}
	\end{array}
	\right.$$

The coefficients \( a_{j,k} \) are defined recursively with respect to \( k \) and \( j \), starting from \( k \) down to 1, independently of \( n \), as follows:

	$$a_{j,k} = \left\{ \begin{array}{lll}
-1,  \text{ when $(j,k)= (1,1)$} \\
	b_j  w_0(j,k), \text{ when $j\in\llbracket  2, k \rrbracket$ and $k\geq j$,}\\
			 \sum\limits_{l=2}^{k} \left( 2^{-2l} w_k(l,k) - a_{l,k}  \right)  2^{2 \eta_k(l-1)}  , \text{ when $j=1$ and $k\geq 2$.}
	\end{array}\right.$$
	where 
	\begin{align*} 
	&\eta_k= \lceil log_2(k) \rceil,\\
	&b_j = 2^{-2}(2^{2(j-1)}-1)^{-1},\\
	&w_{l_k}(j,k)=  \epsilon (k) \epsilon (j)a_{\frac{j}{2},	\frac{k}{2}}^2+{u}(j,k)+ 2v_{l_k}(j,k)\\
\end{align*}
with
	\begin{align*} &u(j,k) =\left\{ \begin{array}{lll}  2\sum_{l= \operatorname{max}(1, j-\frac{k}{2})}^{\operatorname{min}( \lfloor\frac{j-1}{2}\rfloor, \frac{k}{2}-1)} \epsilon (k) a_{l,\frac{k}{2}}a_{j-l, \frac{k}{2}}  & \text{ when } j\geq 3 \text{ and } j \leq k-1 \\
		0 & \text{ otherwise,}
	\end{array}
	\right.\\
	&v_{l_k}(j,k)= \sum\limits^{\lfloor\frac{k-1}{2}\rfloor}_{s=u_{l_k}} \sum_{r=\operatorname{max}(1,j-k+s)}^{\operatorname{min}(j-1,s)}a_{r,s}a_{j-r, k-s}
	\end{align*}
for any $l_k \in \llbracket 0, k \rrbracket$, and $u_{l_k}=\left\{ \begin{array}{lll} 0, \text{ when $l_k=0$} \\
	k-2^{\eta_k -1} ,\text{ when $l_k=k$.} 
	\end{array} \right. $
\end{theorem}

\begin{proof} The case $k=0$ follows from Remark \ref{recursive-remark} (2).  We begin by observing the expression $\operatorname{min}\left(\left\lfloor\frac{k-1}{2}\right\rfloor,2^{n-1}-1\right)= \left\lfloor\frac{k-1}{2}\right\rfloor$, for any $k \in\llbracket 0 , 2^n\rrbracket$.
	
For any $k \in\llbracket  1, 2^n\rrbracket$. Let's introduce the statement $Q(n)$ as follows: 
	$$Q(n): c_{n,2k}= \sum_{j=1}^{k}a_{j,k}2^{2jn},$$ where $a_{j,k}$ is defined recursively as in the statement of the theorem.
	
	We will prove that the statement $Q(n)$ holds for all $n \in \mathbb{N}$ and $k \in \llbracket 1, 2^n \rrbracket$ using induction on $n$.
	 We start with the base case $n=1$ and $k\in \llbracket 1, 2\rrbracket$. In this case, we have $p_1(x)=x^4-4x^2+2$. Therefore, 
	\begin{itemize} 
	\item when $k=1$, since $a_{1,1}=-1$ then $a_{1,1}2^{2j}=-2^{2}= c_{1,2}$; 
 \item  when 
	$k=2$, since $\eta_k=1$, we have $a_{2,2}= \frac{1}{12} a_{1,1}^2= \frac{1}{12}$, and  $$a_{1,2} =2^2( -  a_{2,2}+ 2^{-4} a_{1,1}^2) = -\frac{1}{3}+ \frac{1}{4}= -\frac{1}{12}.$$ So that 
	$$ \sum_{j=1}^{2}a_{j,2}2^{2jn}= \frac{-1}{12}2^{2}+\frac{1}{12}2^{4}= -\frac{1}{3}+ \frac{4}{3}=1 = c_{1,4}.$$
	\end{itemize} 
	This proves the base case. 
	
	Assuming that the statement $Q(n-1)$ holds, our aim is to prove that the statement $Q(n)$ is also true.
	
	Using Lemma \ref{coefficients}, we know that the coefficients can be expressed as follows:
	
	$$c_{n,2k}=\epsilon(k)(c_{n-1,k})^2+2\sum_{s=\operatorname{max}(0, k-2^{n-1})}^{\left\lfloor\frac{k-1}{2}\right\rfloor}c_{n-1,2s}c_{n-1, 2(k-s)}.$$
	
	For any $k \geq 1$, let's define:
	$$\delta_n(k)=
	\begin{cases}
		0 & \text{if } n= \eta_k  \\
		1 & \text{ otherwise. } \\
	\end{cases}$$

	Furthermore, we recall that for all $k \in \mathbb{N}$:
	$$\epsilon (k)=
	\begin{cases}
		0 & \text{if } 2\nmid k\\
		1 & \text{if } 2\mid k \\
	\end{cases}.$$
	
	Using these definitions, we can rewrite the equation as follows:
	$$c_{n,2k}=\epsilon(k)(c_{n-1,k})^2+4\delta_n (k) c_{n-1, 2k}+2\sum_{s=\operatorname{max}(1, k-2^{n-1})}^{\lfloor\frac{k-1}{2}\rfloor}c_{n-1,2s}c_{n-1, 2(k-s)}.$$
	
Indeed, $2^{n-1}+1 \leq k \leq 2^{n}$ if and only if $ \log_2(k) \leq n \leq \log_2(k-1)+1$. Since $n\geq 1$, we have $k\geq 2$, and there exists a unique integer $n$ such that $ \log_2(k) \leq n \leq  \log_2(k-1)+1$, we denote this integer $\eta_k$. 
	
	Let's now define $$f_{n-1,k}=(c_{n-1,k})^2  \text{\ \  and \ \ }g_{n-1,k}=\sum_{s=\operatorname{max}(1, k-2^{n-1})}^{\lfloor\frac{k-1}{2}\rfloor}c_{n-1,2s}c_{n-1,2(k-s)}.$$
	
	By applying the induction hypothesis and the base change of the sums (see \cite[pg  201]{thesis}), we can deduce that
	\begin{align*}
		f_{n-1,k}&=(c_{n-1,k})^2 = \left(\sum_{j=1}^{\frac{k}{2}}a_{j,\frac{k}{2}}2^{-2j}2^{2nj}\right)^2\\
		&=\sum_{j=1}^{\frac{k}{2}}a_{j,\frac{k}{2}}^22^{-4j}2^{2(2nj)}+2\sum_{l=1}^{\frac{k}{2}-1}\sum_{m=l+1}^{\frac{k}{2}}a_{l,\frac{k}{2}}a_{m,\frac{k}{2}}2^{-2(l+m)}2^{2n(l+m)}\\
		&=\sum_{j=2}^{k} (\epsilon (j)a_{\frac{j}{2},\frac{k}{2}}^2+{u}(j,k)) 2^{-2j} 2^{2nj} \
	\end{align*}
	and
	\begin{align*}
		g_{n-1,k} &=\sum_{s=\operatorname{max}(1, k-2^{n-1})}^{\lfloor\frac{k-1}{2}\rfloor}c_{n-1,2s}c_{n-1, 2(k-s)}\\
		&=\sum_{s=\operatorname{max}(1, k-2^{n-1})}^{\lfloor\frac{k-1}{2}\rfloor}\sum_{r=1}^{s}\sum_{t=1}^{k-s}a_{r, s}a_{t, k-s}2^{-2(r+t)}2^{2n(r+t)} \\
		&=\sum_{j=2}^{k}v_{l_k}(j,k) 2^{-2j} 2^{2nj} 
	\end{align*}
	where  $j=r+t$, $l_k= 0$, if $n\neq \eta_k$ and $l_k=k$, otherwise.
	
	Thus, we have:
	\begin{equation}\label{cn2k}
		\begin{aligned}
			c_{n,2k} 
			&=\sum_{j=2}^{k} \left( \epsilon (k) \epsilon (j)a_{\frac{j}{2},\frac{k}{2}}^2+{u}(j,k)+ 2v_{l_k}(j,k) +4\delta_n (k) a_{j,k} \right) 2^{-2j} 2^{2nj}+\delta_n (k)  a_{1,k}2^{2n},
		\end{aligned}
	\end{equation}
Thus, when $n\neq \eta_k$, we deduce the inductive step from Equation (\ref{cn2k}) from the recursive definition of $a_{j,k}'s$, and
$$c_{n,2k}= 	\sum\limits^{k}_{j=1}a_{j,k}2^{2jn}.$$
	
	When $n= \eta_k$, from Equation (\ref{cn2k}), we obtain the induction hypothesis if and only if
	$$\begin{array}{lll} a_{1,k}&=&\sum\limits_{j=2}^{k} \left( \left( \epsilon (k) \epsilon (j)a_{\frac{j}{2},\frac{k}{2}}^2+{u}(j,k)+ 2v_{l_k}(j,k)\right)2^{-2j} -a_{j,k}  \right)2^{2 \eta_k(j-1)} \end{array}$$
	
	Therefore, the inductive step is satisfied, and for all $n$, we have proven the theorem by induction.
\end{proof}

\begin{remark}\label{akk}
	According to the notation of Theorem \ref{coefficients-ai}, in particular:
	for all $k\geq 1$, we have
	\begin{align*}
		a_{k,k} &= b_k \left(\epsilon(k)a_{\frac{k}{2},\frac{k}{2}}^2 + \sum_{s=1}^{\lfloor\frac{k-1}{2}\rfloor}2a_{s,s}a_{k-s, k-s} \right);
			\end{align*}
			and for all $k \geq 2$, we have
	
		\begin{align*}
					a_{2,k} &= b_2 \left( \epsilon(k)a_{1,\frac{k}{2}}^2 +\sum_{s=1}^{\lfloor\frac{k-1}{2}\rfloor}2a_{1,s}a_{1, k-s}\right). \\
	\end{align*}

\end{remark}
We compute a few values for the $a_{i,j}$'s in the next example. 
\begin{example} 
We demonstrated in the base case of the induction proof that the formulas yield the values of \(a_{1,1}\), \(a_{1,2}\), and \(a_{2,2}\), which were also calculated in the discussion preceding Theorem \ref{coefficients-ai}. Indeed, we found $a_{1,1}=-1$, $a_{1,2}= -\frac{1}{12}$ and $a_{2,2}= \frac{1}{12}$. 

We now compute $a_{i,j}$ for $j \in \{ 3, 4\}$ and $i\in \{1, \cdots, j\}$. \\
For $j=3$, we have
 $$\begin{array}{l}   \begin{array}{lll}   a_{3,3} &=& b_3w_0(3,3)=  2 b_3 v_0(3, 3)=2 b_3 a_{1,1} a_{2,2}=  -\frac{2^{-2} ( 2^4 -1)^{-1}}{6} =   -\frac{1}{360}\end{array}\\  
\begin{array}{lll}   a_{2, 3}& =& b_2w_0(2,3)=  2 b_2 v_0(2, 3)=2b_2 a_{1,1} a_{1,2}=  \frac{1}{72}\end{array}\\
\begin{array}{lll}a_{1,3} &=&(2^{-4} w_3(2,3) - a_{2,3})  2^{2 \eta_3}  +  (2^{-6} w_3(3,3) - a_{3,3})2^{4 \eta_3} \\
&=&    2v_3(2,3) - \frac{2^{4}}{72}   +  2^{3} v_3(3,3) +\frac{ 2^{8}}{ 360}   =    2a_{1,1} a_{1,2} - \frac{2^{4}}{72}   +  2^{3} a_{1,1} a_{2,2}+\frac{ 2^{8}}{ 360}  \\
&=&   \frac{1}{6} - \frac{2^{4}}{72}   +  -\frac{2}{3}+\frac{ 2^{8}}{ 360} = -\frac{1}{90},\end{array} \end{array}$$
 For $j=4$, we have
$$\begin{array}{l} \begin{array}{lll}    a_{4,4} &=&  b_4w_0(4,4)=2^{-2}(2^6-1)^{-1}(a_{2,2}^2+2a_{1,1}a_{3,3})=   2^{-2}(2^6-1)^{-1}(\frac{1}{144}+\frac{2}{360})=\frac{1}{20160},\end{array}\\
\begin{array}{lll}   a_{3,4} &=&   b_3w_0(3,4)=b_3(u(3,4)+2v_0(3,4))=b_3(2a_{1,2}a_{2,2}+2a_{1,1}a_{2,3})

=\frac{-1}{1440},\end{array}\\
\begin{array}{lll}   a_{2,4} &=&   b_2w_0(2,4)=b_2(a_{1,2}^2+2v_0(2,4))=b_2(a_{1,2}^2+2a_{1,1}a_{1,3})
	
=\frac{7}{2880},\end{array}\\
\begin{array}{lll}   a_{1,4} &=&(2^{-4}w_4(2,4)-a_{2,4})2^4+(2^{-6}w_4(3,4)-a_{3,4})2^8+(2^{-8}w_4(4,4)-a_{4,4})2^{12}\\
&=&(2^{-4}a_{1,2}^2-a_{2,4})2^4+(2^{-6}u(3,4)-a_{3,4})2^8+(2^{-8}a_{2,2}^2-a_{4,4})2^{12}\\
&=&(2^{-4}a_{1,2}^2-a_{2,4})2^4+(2^{-6}(2a_{1,2}a_{2,2})-a_{3,4})2^8+(2^{-8}a_{2,2}^2-a_{4,4})2^{12}\\
&=&  (\frac{2^{-4}}{144}-\frac{7}{2880})2^4+(\frac{-2^{-5}}{144}+\frac{1}{1440})2^8+(\frac{2^{-8}}{144}-\frac{1}{20160})2^{12}= \frac{-1}{560}\end{array} \end{array}$$ 
\end{example}

\subsection{A different recursive method for calculating the \texorpdfstring{$a_{i,j}$'s}{Lg}}
In the following lemma, we present a different method of computing the coefficients $a_{j,k}$ as defined in Theorem \ref{coefficients-ai}. This approach offers enhanced efficiency by reducing the number of required operations. To achieve this, we make use of the values of $c_{n, 2k}$ derived from Theorem \ref{coefficients-field} and apply the algorithm described in \cite[pg.895]{pereyra-vandermonde}.
\begin{proposition}\label{vandermond}
	Let $n\geq 1$ and $k$ be positive integers. We set $\eta_k = \lceil \operatorname{log}_2 (k)\rceil$. For each $j \in \llbracket1,  k \rrbracket$, we define $\ell_k (j) = j -1 + \eta_k$. According to the notation of Theorem \ref{coefficients-ai}, for each $i\in \llbracket1,  k \rrbracket$, we define the vectors:
	\begin{itemize} 
	\item ${\nu}^{(i)}=(\nu^{(i)}_j)_{j\in \llbracket1,k \rrbracket}$ is defined recursively as follows:
	\begin{itemize} 
	\item $\nu^{(1)}= \left( \frac{c_{\ell_k (j), 2k}}{2^{2\ell_k (j)}} \right)_{j\in \llbracket1,k \rrbracket} $;
	\item for $i \in \llbracket 1,  k-1\rrbracket$ and $j \in \llbracket1, k\rrbracket$, we define $\nu_{j}^{(i+1)}$ as follows:
	$$
	\nu_{j}^{(i+1)} =
	\begin{cases}
		\frac{\nu_j^{(i)}-\nu_{j-1}^{(i)}}{2^{2\ell_k (j)}-2^{2(\ell_k (j)-i)}} & \text{if } j\in \llbracket i+1, k\rrbracket \\
		\nu_{j}^{(i)} & \text{if } j\in \llbracket 1, i\rrbracket \\
	\end{cases}
	$$
	\end{itemize}
	\item ${a}^{(i)}=(a^{(i)}_j)_{j\in \llbracket1,k \rrbracket} $ is defined recursively as follows:
		\begin{itemize} 
	\item $a^{(k)}={\nu^{(k)}}$;
	\item for $i \in \llbracket1, k-1\rrbracket$ and $j \in \llbracket 1, k\rrbracket$, we define $a_{j}^{(i)}$ as follows:
	$$
	a_{j}^{(i)} =
	\begin{cases}
		a_{j}^{(i+1)}-2^{2\ell_k (i)}a_{j+1}^{(i+1)} & \text{if } j\in \llbracket i,  k-1\rrbracket \\
		a_{j}^{(i+1)} & \text{if } j\leq i-1 \ \text{or} \ j=k\\
	\end{cases}
	$$
	\end{itemize}
	\end{itemize}
	Then, we have, $ a_{j,k}=a_j^{(1)}$ for all $j\in \llbracket1, k\rrbracket$.
\end{proposition}
  \begin{proof}
	For $n\in \llbracket1,k\rrbracket$, we have by Theorem \ref{coefficients-ai} the following Vandermonde matrix equation
      \[
\begin{bmatrix}
    1 & 2^{2\eta_k} & 2^{4\eta_k} &\cdots  & 2^{2\eta_k(k-2)} & 2^{2\eta_k(k-1)} \\
    1 & 2^{2(1+\eta_k)} & 2^{4(1+\eta_k)} & \cdots  & 2^{2(\eta_k+1)(k-2)} & 2^{2(\eta_k+1)(k-1)} \\
    \vdots & \vdots & \vdots & \vdots & \vdots & \vdots  \\
    1 & 2^{2(k-1+\eta_k)} & 2^{4(k-1+\eta_k)}&  \cdots & 2^{2(k-1+\eta_k)(k-2)} & 2^{2(k-1+\eta_k)(k-1)}\\
\end{bmatrix}
\begin{bmatrix}
    a_{1,k} \\
  a_{2,k}\\
    \vdots\\
   a_{k,k}\\
\end{bmatrix}
=
\begin{bmatrix}
   \frac{c_{\eta_k, 2k}}{2^{2\eta_k}} \\
  \frac{c_{1+\eta_k, 2k}}{2^{2(1+\eta_k)}} \\
\vdots \\
\frac{c_{k-1+\eta_k,2k}}{2^{2(k-1+\eta_k)} }
\end{bmatrix}
\]

Therefore, the Lemma follows from the algorithm provided in \cite[p.895]{pereyra-vandermonde}.
\end{proof}

\begin{example}
  \begin{enumerate}
  \item {\bf Case $k=3$.} We can represent the Vandermonde matrix equation as follows:

        \[
\begin{bmatrix}
    1 & 2^4 & 2^8 \\ \\
    1 & 2^{6} & 2^{12}\\  \\
        1 & 2^{8} & 2^{16}  \\ 
\end{bmatrix}
\begin{bmatrix}
    a_{1,3} \\ \\
  a_{2,3}\\ \\
   a_{3,3}\\
\end{bmatrix}
=
\begin{bmatrix}
  \frac{c_{2,6}}{2^4} \\ \\
\frac{c_{3,6}}{2^{6}}\\\\
   \frac{c_{4, 6}}{2^8} \\ 
\end{bmatrix} = \begin{bmatrix}
  \frac{-1}{2} \\ \\
\frac{-21}{2}\\ \\
   -\frac{357}{2} \\ 
\end{bmatrix} = \nu^{(1)}
\]

By applying the recursive formulas, we can compute the following results:
$$\begin{array}{cccc}
 \nu^{(2)}  &= \left[\nu_1^{(1)},      \frac{ \nu_2^{(1)}- \nu_1^{(1)}}{2^6-2^4},  \frac{\nu_3^{(1)}- \nu_2^{(1)}}{2^8-2^6}\right]^\intercal= \left[ -\frac{1}{2} ,      -\frac{5}{24},  -\frac{7}{8} \right]^\intercal;
   \end{array}$$
$$\begin{array}{cccc}
    {a}^{(3)}= {\nu^{(3)}}  &= \left[ \nu_1^{(2)}, \nu_2^{(2)}, \frac{\nu_3^{(2)}- \nu_2^{(2)}}{2^8-2^4}\right]^\intercal =  \left[  -\frac{1}{2} ,   -\frac{5}{24}, -\frac{1}{360}\right]^\intercal;
\end{array}$$
 $$\begin{array}{cccc}
      {a}^{(2)}   & = \left[ {a}_{1}^{(3)}, {a}_{2}^{(3)}-2^6 {a}_{3}^{(3)}, {a}_{3}^{(3)}\right]^\intercal = \left[  -\frac{1}{2}, -\frac{11}{360}, -\frac{1}{360}\right]^\intercal;
 \end{array}$$
 $$\begin{array}{cccc}
      {a}^{(1)}   & = \left[   a_{1,3},  a_{2,3},   a_{3,3}\right]^\intercal=  \left[  {a}_{1}^{(2)}-2^4 {a}_{2}^{(2)}, {a}_{2}^{(2)}-2^4 {a}_{3}^{(2)}, {a}_{3}^{(2)}\right]^\intercal = \left[  -\frac{1}{90}, \frac{1}{72},  -\frac{1}{360}\right]^\intercal
 \end{array}.$$
 \item {\bf Case $k=4$.} The Vandermonde matrix is as follows:
 \[\begin{bmatrix}
 	1& 2^4 & 2^8 & 2^{12}\\ \\
 	1& 2^6 & 2^{12} & 2^{18}\\ \\
 	1&2^8 & 2^{16} & 2^{24} \\ \\ 
	 	1 & 2^{10} & 2^{20} & 2^{30}\\ 
 \end{bmatrix}
 \begin{bmatrix}
 	a_{1,4}\\ \\
 	a_{2, 4}\\ \\
 	a_{3, 4}\\ \\
 	a_{4, 4}
 \end{bmatrix}
 =
 \begin{bmatrix}
 	\frac{c_{2, 8}}{2^4}\\ \\
 	\frac{c_{3, 8}}{2^6}\\ \\
 	\frac{c_{4, 8}}{2^8} \\ \\ 
	\frac{c_{5, 8}}{2^{10}}
 \end{bmatrix}
 =
 \begin{bmatrix}
 	\frac{1}{16}\\ \\
 	\frac{165}{16}\\ \\
 	\frac{12597}{16} \\ \\ 
	\frac{840565}{16} 
 \end{bmatrix}
 =\nu^{(1)}
 \]
 Applying the recursive formulas in Lemma \ref{vandermond}, we obtain the following matrices:
 $$\nu^{(2)}=\left[ \nu_1^{(1)}, \frac{\nu_2^{(1)}-\nu_1^{(1)}}{2^6-2^4}, \frac{\nu_3^{(1)}-\nu_2^{(1)}}{2^8-2^6}, \frac{\nu_4^{(1)}-\nu_3^{(1)}}{2^{10}-2^8}\right]^\intercal=\left[ 	\frac{1}{16}, \frac{41}{192}, \frac{259}{64}, \frac{12937}{192}\right]^\intercal;
 $$
 $$\nu^{(3)}=\left[\nu_1^{(2)}, \nu_2^{(2)},\frac{\nu_3^{(2)}-\nu_2^{(2)}}{2^8-2^4}, \frac{\nu_4^{(2)}-\nu_3^{(2)}}{2^{10}-2^6} \right]^\intercal=\left[\frac{1}{16}, \frac{41}{192},\frac{23}{1440},\frac{19}{288}\right]^\intercal;
 $$
 $$\nu^{(4)}=\left[\nu_1^{(3)}, \nu_2^{(3)}, \nu_3^{(3)}, \frac{\nu_4^{(3)}-\nu_3^{(3)}}{2^{10}-2^4}\right]^\intercal= \left[\frac{1}{16}, \frac{41}{192},\frac{23}{1440}, \frac{1}{20160}\right]^\intercal=a^{(4)};$$
 $$a^{(3)}=\left[a_1^{(4)}, a_2^{(4)}, a_3^{(4)}-2^8a_4^{(4)}, a_4^{(4)}\right]^\intercal=\left[\frac{1}{16}, \frac{41}{192}, \frac{11}{3360}, \frac{1}{20160}\right]^\intercal;$$
 $$a^{(2)}=\left[a_1^{(3)}, a_2^{(3)}-2^6a_3^{(3)}, a_3^{(3)}-2^6a_4^{(3)}, a_4^{(3)}\right]^\intercal=\left[\frac{1}{16}, \frac{9}{2240},  \frac{1}{10080}, \frac{1}{20160}\right]^\intercal;
 $$
 $$a^{(1)}=\left[a_1^{(2)}-2^4a_2^{(2)}, a_2^{(2)}-2^4a_3^{(2)}, a_3-2^{4}a_4^{2}, a_4^{(2)}\right]=\left[-\frac{1}{560}, \frac{7}{2880}, -\frac{1}{1440}, \frac{1}{20160}\right]^\intercal;$$
 As a result, we obtain $$a_{1,4}=-\frac{1}{560}, \ a_{2, 4}=\frac{7}{2880}, \ a_{3, 4}=-\frac{1}{1440}, \ a_{4,4}=\frac{1}{20160}.$$  
 \end{enumerate}
\end{example}

\section{Representation for the coefficient \texorpdfstring{$a_{k,k}$}{Lg} as weighted Catalan numbers}
In this section, we show an interesting connection between the values of the constants $a_{k,k}$ computed in Theorem \ref{coefficients-ai} and Lemma \ref{vandermond} and the famous Catalan numbers, revealing how they are related. We use the notation introduced in Theorem \ref{coefficients-ai}. We introduce the following definition regarding ordered trees, which we will use to express a formula for the $a_{k,k}$s.
\begin{definition}\label{def:gk}
 Let $k \in \mathbb{N}\setminus \{0\}$.   The set $\mathcal{T}_k$ consists of labeled ordered tree defined as follows:
    \begin{itemize}
        \item The nodes are labeled with numbers ranging from $1$ to $k$. We write  this nodes as $\scalebox{0.8}{\begin{tikzpicture}[level distance=1.5cm,
     	level 1/.style={sibling distance=3.5cm},
     	level 2/.style={sibling distance=1cm},
     	tree node/.style={circle,draw},
     	every child node/.style={tree node}]
     	
     	\node[tree node] (Root)  {1};
     \end{tikzpicture} } $, $\cdots$, $\scalebox{0.8}{\begin{tikzpicture}[level distance=1.5cm,
     	level 1/.style={sibling distance=3.5cm},
     	level 2/.style={sibling distance=1cm},
     	tree node/.style={circle,draw},
     	every child node/.style={tree node}]
     	
     	\node[tree node] (Root)  {k};
     \end{tikzpicture} } $.

 \item \begin{itemize}     
\item Each node labeled \( 1 \) is connected by an edge to exactly one parent node, which is labeled with an integer strictly greater than \( 1 \).
\item There is only one root, labeled \( k \), and this root is connected by an edge to exactly two ordered child nodes, whose labels add up to \( k \).
\item For each node labeled \( v \) with \( 2 < v < k \), there are exactly three nodes connected to it: one parent node labeled with an integer strictly greater than \( v \), and two ordered child nodes whose labels sum to \( v \).
        \end{itemize}
        In summary, all the possible edges are included in the following trees: 
        $$\begin{array}{ccc} \scalebox{0.8}
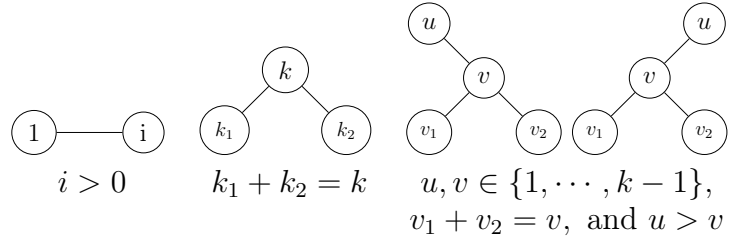
{\begin{tikzpicture}
				[scale=.9,auto=center,every node/.style={circle, draw}] \node (a1) at (0, 0){1} ;
				\node (a2) at (2, 0) {i} ;
				\draw (a1) -- (a2);
			\end{tikzpicture} }&  \scalebox{0.8}{\begin{tikzpicture}  
				[every node/.style={circle,draw}] 
				
				\node (a1) at (0,0) {\scalebox{0.8}{$k_1$}};  
				\node (a2) at (1,1)  {$k$}; 
				\node (a3) at (2,0)  {\scalebox{0.8}{$k_2$}};  
				
				\draw (a1) -- (a2);  
				\draw (a2) -- (a3);  
			\end{tikzpicture} } & \scalebox{0.8}{\begin{tikzpicture}  
				[scale=.9,auto=center,every node/.style={circle, draw}] 
				
				\node (a1) at (0,0) {\scalebox{0.8}{$v_1$}};  
				\node (a2) at (1,1)  {$v$}; 
				\node (a3) at (2,0)  {\scalebox{0.8}{$v_2$}}; 
				\node (a4) at (0,2)  {$u$}; 
				
				\draw (a1) -- (a2);  
				\draw (a2) -- (a3);
				\draw (a2) -- (a4); 
			\end{tikzpicture}}  \text{   }  \scalebox{0.8}{\begin{tikzpicture}  
				[scale=.9,auto=center,every node/.style={circle, draw}] 
				
				\node (a1) at (0,0) {\scalebox{0.8}{$v_1$}};  
				\node (a2) at (1,1)  {$v$}; 
				\node (a3) at (2,0)  {\scalebox{0.8}{$v_2$}}; 
				\node (a4) at (2,2)  {$u$}; 
				
				\draw (a1) -- (a2);  
				\draw (a2) -- (a3);
				\draw (a2) -- (a4); 
			\end{tikzpicture}} \\ 
			i > 0 & k_1+k_2=k &  u, v \in \{1, \cdots, k-1\}, \\ && v_1+v_2=v, \text{ and } u>v \end{array}$$
    \end{itemize}
\end{definition}

In the next example, we describe $\mathcal{T}_k$ for $k \in \{ 1, 2, 3,4\}$. 
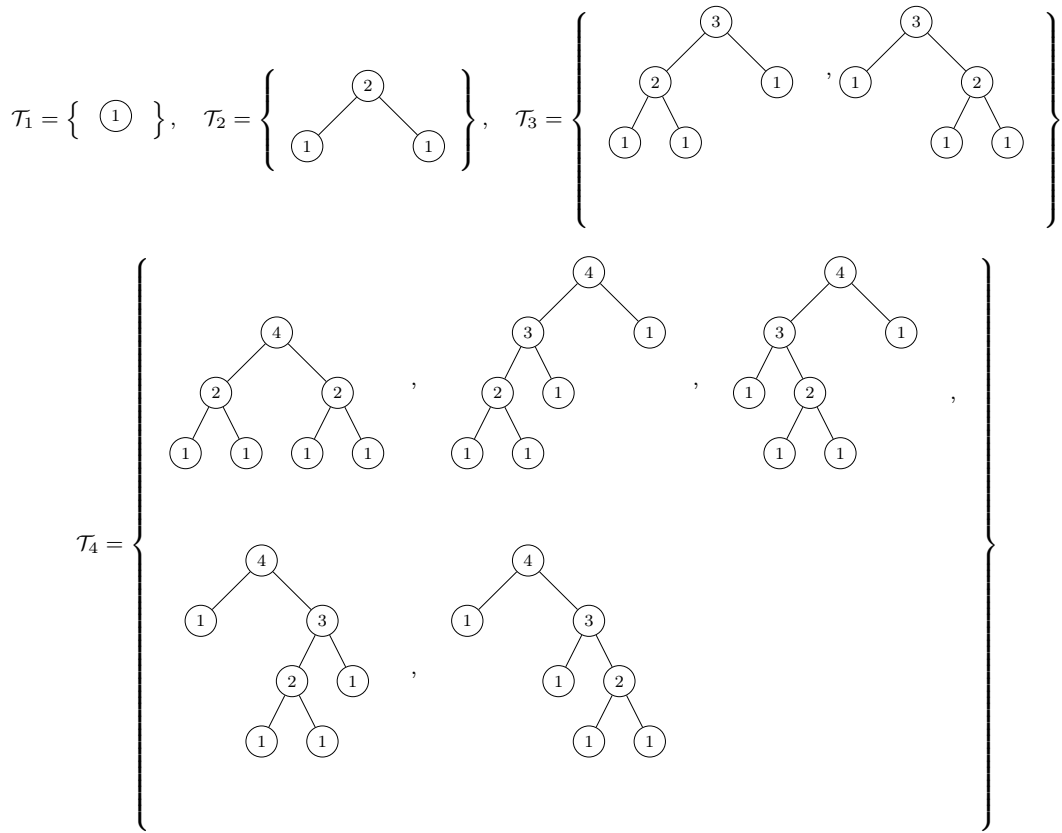
  \begin{example}
$$\tiny
   \begin{array}{ccc}
     \mathcal{T}_1 =\left\{ \begin{array}{c} \scalebox{0.8}{ \begin{tikzpicture}[level distance=1.5cm,
     	level 1/.style={sibling distance=3.5cm},
     	level 2/.style={sibling distance=1cm},
     	tree node/.style={circle,draw},
     	every child node/.style={tree node}]
     	
     	\node[tree node] (Root)  {1};
     \end{tikzpicture} } \end{array} \right\},  &  \mathcal{T}_2= \left\{\begin{array}{c} \scalebox{0.8}{\begin{tikzpicture}[level distance=1cm,
     level 1/.style={sibling distance=2cm},
     level 2/.style={sibling distance=1cm},
     tree node/.style={circle,draw},
     every child node/.style={tree node}]
     
     \node[tree node] (Root) {2}
     child {
     	node {1} 
     }
     child {	node {1}
     };	
     \end{tikzpicture}} \end{array} \right\},  & \mathcal{T}_3= \left\{  \begin{array}{ccc} \scalebox{0.8}{ \begin{tikzpicture}[level distance=1cm,
     level 1/.style={sibling distance=2cm},
     level 2/.style={sibling distance=1cm},
     tree node/.style={circle,draw},
     every child node/.style={tree node}]
     
     \node[tree node] (Root) {3}
     child {
     	node {2} 
     	child{node {1}}
     	child {node{1}}
     }
     child {	node {1}
     };	
     \end{tikzpicture} } & \begin{minipage}{0.1cm}  , \\ \text{} \\  \text{} \\   \text{} \\  \text{} \\  \text{} \\ \end{minipage} \scalebox{0.8}{ \begin{tikzpicture}[level distance=1cm,
  level 1/.style={sibling distance=2cm},
  level 2/.style={sibling distance=1cm},
  tree node/.style={circle,draw},
  every child node/.style={tree node}]
  
  \node[tree node] (Root) {3}
  child {	node {1}
  }
  child {
  	node {2} 
  	child{node {1}}
  	child {node{1}}
  };	
  \end{tikzpicture} } \end{array}
  \right\} 
   \end{array}$$

   $$
   \tiny
   \begin{array}{ccc}
        \mathcal{T}_4=  \left\{ \begin{array}{cccccccc} 
\scalebox{0.8}{ \begin{tikzpicture}[level distance=1cm,
	level 1/.style={sibling distance=2cm},
	level 2/.style={sibling distance=1cm},
	tree node/.style={circle,draw},
	every child node/.style={tree node}]
	
	\node[tree node] (Root) {4}
	child {
		node {2} 
		child{node {1}}
		child {node{1}}
	}
	child {
		node {2} 
		child{node {1}}
		child {node{1}}
	};	
\end{tikzpicture}}  & \begin{minipage}{0.1cm}  , \\ \text{} \\  \text{} \\   \text{} \\  \text{} \\  \text{} \\ \end{minipage} &

 \scalebox{0.8}{ \begin{tikzpicture}[level distance=1cm,
  	level 1/.style={sibling distance=2cm},
  	level 2/.style={sibling distance=1cm},
  	tree node/.style={circle,draw},
  	every child node/.style={tree node}]
  	
  	\node[tree node] (Root) {4}
  	child {
  		node {3} 
  		child{node {2}
  			child {node {1}}
  			child{node{1}}
  		}	
  		child {
  			node{1}}
  	}
  	child {
  		node {1} };
  \end{tikzpicture}}
& \begin{minipage}{0.1cm}  , \\ \text{} \\  \text{} \\   \text{} \\  \text{} \\ \text{} \\ \end{minipage} & 

\scalebox{0.8}{ \begin{tikzpicture}[level distance=1cm,
	level 1/.style={sibling distance=2cm},
	level 2/.style={sibling distance=1cm},
	tree node/.style={circle,draw},
	every child node/.style={tree node}]
	
	\node[tree node] (Root) {4}
	child {
		node {3} 	
		child {
			node{1}}
		child{node {2}
			child {node {1}}
			child{node{1}}
		}
	}
	child {
		node {1} };
\end{tikzpicture}  }
& \begin{minipage}{0.2cm}  , \\ \text{} \\  \text{} \\   \text{} \\  \text{} \\\end{minipage}  \\

   \scalebox{0.8}{ \begin{tikzpicture}[level distance=1cm,
   	level 1/.style={sibling distance=2cm},
   	level 2/.style={sibling distance=1cm},
   	tree node/.style={circle,draw},
   	every child node/.style={tree node}]
   	
   	\node[tree node] (Root) {4}
   	child {
   		node {1} }
   	child {
   		node {3} 
   		child{node {2}
   			child {node {1}}
   			child{node{1}}
   		}	
   		child {
   			node{1}}
   	};
   \end{tikzpicture}}
   
 &  \begin{minipage}{0.1cm}  , \\ \text{} \\  \text{} \\   \text{} \\  \text{} \\  \text{} \\\end{minipage}& 
  \scalebox{0.8}{ \begin{tikzpicture}[level distance=1cm,
  	level 1/.style={sibling distance=2cm},
  	level 2/.style={sibling distance=1cm},
  	tree node/.style={circle,draw},
  	every child node/.style={tree node}]
  	\node[tree node] (Root) {4}
  	child {
  		node {1} }
  	child {
  		node {3} 
  		child {
  			node{1}}
  		child{node {2}
  			child {node {1}}
  			child{node{1}}
  		}
  	};
  \end{tikzpicture}} \end{array} \right\}
   \end{array}$$
  \end{example}

\begin{remark}\label{Gk}

\begin{enumerate}
\item The cardinality of \( \mathcal{T}_k \) is the \((k-1)^{\text{th}}\) Catalan number.
\item For every ordered tree in \( \mathcal{T}_k \), removing the top node labeled \( k \) and its associated edges splits the tree into two distinct ordered trees:  \( S\in \mathcal{T}_s \) and \( R \in \mathcal{T}_t \) with \( s + t = k \). Conversely, given two ordered trees \( S \in \mathcal{T}_s \) and \( R \in \mathcal{T}_t \) with \( s + t = k \), we can construct distinct trees in \( \mathcal{T}_k \) as follows:  
\begin{itemize}
\item If \( s \neq t \), we can form two distinct trees by adding a vertex labeled \( k \), connecting it to the root of \( S \) with an edge labeled \( s \), and to the root of \( R \) with an edge labeled \( t \).  
\item If \( s = t \), we can form one tree as described earlier.  
\end{itemize}
\end{enumerate}

\end{remark}
In the following definition, we introduce some notation. 
\begin{definition}\label{def:notat}
Let \( k \in \mathbb{N} \setminus \{0\} \) and \( T \) be an ordered tree in the set \( \mathcal{T}_k \).
\begin{enumerate}
\item We denote by \( V_T \) the set of integers labeling the nodes of \( T \).
\item For any integer \( v \in V_T \), we let \( \delta_{v,T} \) denote the number of times the label \( v \) appears on the nodes of \( T \).
\end{enumerate}
\end{definition}

We now prove that the values \( a_{k,k} \) are also weighted Catalan numbers.
\begin{proposition}\label{wc}
Let \( k \in \mathbb{N} \setminus \{0\} \).  
We have the following expression for \( a_{k,k} \):

\[
a_{k, k} = \sum_{T \in \mathcal{T}_k} \prod_{v \in V_T} b_v^{\delta_{v,T}},
\]

where \( b_1 = -1 \) and \( b_v = 2^{-2} \left( 2^{2(v-1)} - 1 \right)^{-1} \) for all \( v > 1 \).
\end{proposition}
  \begin{proof}
We will prove the lemma by induction on \( k \). We begin with the base case when \( k = 1 \). In this case, \( \mathcal{T}_1 \) contains only one ordered tree \( T \) with a single vertex labeled \( v = 1 \), so \( V_T = \{1\} \) and \( \delta_1 = 1 \). This gives:

\[
a_{1,1} = b_1 = \sum_{T \in \mathcal{T}_1} \prod_{v \in V_T} b_v^{\delta_{v,T}},
\]

which proves the base case.

Now, assume the statement holds for any $l<k$. We aim to prove it for \( k \). By Remark \ref{akk}, we have:

\[
a_{k,k} = b_k \left( \epsilon(k) a_{\frac{k}{2}, \frac{k}{2}}^2 + 2 \sum_{s=1}^{\left\lfloor \frac{k-1}{2} \right\rfloor} a_{s,s} a_{k-s, k-s} \right).
\]

Since \( \frac{k}{2} \), \( s \), and \( k - s \) are all strictly less than \( k \), we can apply the induction hypothesis:

\[
a_{k,k} = \epsilon(k) \left( \sum_{T \in \mathcal{T}_{\frac{k}{2}}} \prod_{v \in V_T} b_v^{\delta_{v,T}} \right)^2 b_k + 2 \sum_{s=1}^{\left\lfloor \frac{k-1}{2} \right\rfloor} \sum_{T \in \mathcal{T}_s} \sum_{R \in \mathcal{T}_{k-s}} \prod_{v \in V_T} \prod_{w \in V_R} b_v^{\delta_{v,T}} b_w^{\delta_{w,R}} b_k.
\]

The induction step is then concluded by applying Remark \ref{Gk}, (2).
  \end{proof}

\begin{remark}
    Note that \( b_v \) represents the weight of a node labeled \( v \) in the tree \( T \), and \( \prod_{v \in V_T} b_v^{\delta_{v,T}} \) is the total weight of the tree \( T \).
\end{remark}

  We use the formula in Proposition \ref{wc} to compute $a_{k,k}$ for $k \in \{ 1, 2, 3, 4 ,5\}$.
 \begin{example} 
We have $$a_{11}=b_1$$
$$a_{22}=b_1^2b_2$$
$$a_{33}=2b_1^3b_2b_3$$
$$a_{44}=b_1^4b_2^2b_4+2^2b_1^4b_2b_3b_4$$
$$ a_{5,5}=2b_1^5b_2^2b_4b_5+2^3b_1^5b_2b_3b_4b_5+2^2b_1^5b_2^2b_3b_5$$
where \( b_1 = -1 \) and \( b_v = 2^{-2} \left( 2^{2(v-1)} - 1 \right)^{-1} \) for all \( v > 1 \).
 \end{example}


	
\bibliographystyle{Abbrv}

 \end{document}